\documentclass[12pt,table]{amsart}
\usepackage{amsmath,amsfonts,amssymb,amsthm,amstext,hyperref,verbatim,lmodern,ytableau,textcomp,color,young,youngtab,tikz,subfiles,tikz-cd,bbold}
\usepackage{xcolor}
\usepackage{enumerate,verbatim,mathtools,graphicx,pgf}
\usepackage{mathrsfs}
\usepackage[mathscr]{euscript}
\usepackage[utf8]{inputenc}

\usepackage{calligra}
\usepackage{setspace,fancyhdr}
    \newtheorem{thm}{Theorem}[section]
    \newtheorem{lem}[thm]{Lemma}

\theoremstyle{definition}

    \newtheorem{rem}[thm]{Remark}
    \newtheorem{exmp}[thm]{Example}

\numberwithin{equation}{section}
\usepackage[square,sort&compress,comma,numbers]{natbib}
\usepackage{hyperref}

\definecolor{amber}{rgb}{1.0, 0.75, 0.0}
\definecolor{ao}{rgb}{0.0, 0.50, 0.0}
\usepackage[a4paper,top=3cm,bottom=2cm,left=3cm,right=3cm,marginparwidth=1.75cm]{geometry}
\DeclareMathOperator{\cb}{
\arraycolsep=1pt\def\arraystretch{0.6}
\begin{array}{c}
\circ\\
\bullet
\end{array}
}
\DeclareMathOperator{\bc}{
\arraycolsep=1pt\def\arraystretch{0.6}
\begin{array}{c}
\bullet\\
\circ
\end{array}
}
\DeclareMathOperator{\bb}{
\arraycolsep=1pt\def\arraystretch{0.6}
\begin{array}{c}
\bullet\\
\bullet
\end{array}
}
\DeclareMathOperator{\cc}{
\arraycolsep=1pt\def\arraystretch{0.6}
\begin{array}{c}
\circ\\
\circ
\end{array}
}

\usepackage{amsmath}
\usepackage{graphicx}
\usepackage[colorinlistoftodos]{todonotes}

\title{Correlations in the multispecies 
 PASEP on a ring}

\author{Nimisha Pahuja}
    \address{Department of Mathematics, Indian Institute of Science, Bangalore - 560012}
    \email{nimishap@iisc.ac.in}

\DeclareMathOperator{\iden}{\langle 1^n \rangle}
\begin{document}

\begin{abstract}
 
Ayyer and Linusson studied correlations in the multispecies TASEP on a ring (Trans AMS, 2017) using a combinatorial analysis of the multiline queues construction defined by Ferrari and Martin (AOP, 2008). It is natural to explore whether an analogous application of appropriate multiline queues could give similar results for the partially asymmetric case. In this paper, we solve this problem of correlations of adjacent particles on the first two sites in the multispecies PASEP on a finite ring. We use the multiline processes defined by Martin (EJP, 2020), the dynamics of which also depend on the asymmetry parameter $q$, to compute the correlations.
\end{abstract}

%
%
\maketitle


\section{Introduction}\label{sec:introPASEP}

  The \textit{asymmetric simple exclusion process} or ASEP is a fundamental stochastic model that describes the probabilistic movement of particles on a one-dimensional lattice, where each site can be occupied by at most one particle~\cite{Intro_Liggett}. This model has been extensively studied in many different settings, and many of its properties are of interest to probabilists, combinatorialists and statistical physicists. One important property of the ASEP that has received significant attention is the correlation between adjacent particles in the stationary distribution of the process~\cite{aaslinusson,  Intro_AAV,ayyerlinusson, uchiyama2005correlation}.

   The standard ASEP on a ring is a Markov process where the particles move according to certain rules, such that they can hop to a neighbouring empty site. The rate at which these transitions occur is given by  $q \in [0,1)$ when the particle moves towards the left (or counterclockwise), and $1$ when the particle jumps rightwards (or clockwise). In a multispecies asymmetric exclusion process, the particles have a certain hierarchy, characterized by an integer labelling on each particle. This label is known as the \textit{type} or the \textit{species} of the particle. Each particle carries an exponential clock which rings with a rate $1$, and the particle tries to jump to its neighbouring site whenever the clock rings. If a particle labeled $i$ attempts to hop to the neighboring site on the right occupied by a particle labeled $j$, the particles can exchange places according to the following specified interaction rules.
     \[ij \rightarrow ji \text{ with rate }
     \begin{cases}
         1, &\text{if }i>j,\\
         q, &\text{if }i<j,
     \end{cases}\]
 where $q \in [0,1)$. Depending on the value of $q$, a multispecies ASEP can be of one of the two kinds; \textit{totally asymmetric simple exclusion process} or TASEP where $q=0$ and \textit{partially asymmetric simple exclusion process} or PASEP where $0 < q < 1$. The rate $q$ is called the \textit{asymmetry parameter} of the ASEP. In this article, we study the correlations in the multispecies PASEP on a ring.
 
 The steady-state distribution of the two-species asymmetric exclusion process on a \textit{ring} was obtained by Derrida, Janowsky, Lebowitz and Speer~\cite{Intro_DJLS} using a technique called \textit{matrix product ansatz}. Ferrari and Martin~\cite{ferrarimartin,ferrarimartin2} gave a probabilistic solution by constructing \textit{multiline queues} as a device to study TASEP with multiple species. The construction of Ferrari and Martin inspired later works where their algorithm was transformed into a matrix product representation of the multispecies TASEP~\cite{Intro_AAMP,evans2009matrix}. The exact solution for the stationary state measure for the PASEP on a ring with multiple species was found by Prolhac, Evans and Mallick~\cite{prolhac}. This was achieved by extending the matrix-product representation from \cite{evans2009matrix} to $q>0$.

 Recently, Martin~\cite{jamesmartin} studied the stationary distribution of the multispecies ASEP on a finite ring combinatorially and developed a method to sample exactly from the stationary distribution. This analysis was done by using queuing systems which are constructed recursively and can be seen as \textit{multiline diagrams} or \textit{multiline queues}. The connections of exclusion processes to Schubert polynomials, Macdonald polynomials, and orthogonal polynomials have also been explored by various authors \cite{Intro_Cantini, Intro_cgw, Intro_CMW_Macdonald}. In a recent work, Corteel, Madelshtam and Williams~\cite{Intro_CMW_Macdonald} gave an independent proof of Martin’s result, and used the multiline queues to give a new combinatorial formula for Macdonald polynomials. 
  
  Martin~\cite{jamesmartin} raised a question of whether the combinatorial analysis done by Ayyer and Linusson to compute closed-form expressions for the multi-point probabilities in multispecies TASEP~\cite{ayyerlinusson} can be generalised to the case $q>0$. In this paper, we compute the two-point correlations in the stationary distribution of the multispecies PASEP on a finite ring. To carry out this investigation, we use the \textit{multiline process} described in~\cite{jamesmartin} and the procedure of lumping~\cite{LPW_book} to transform the study of the stationary distribution of the multispecies PASEP into that of the stationary distribution of the multiline process. We have also used this technique earlier to study the correlations in multispecies continuous TASEP on a ring~\cite{Fullversion}. We use an algorithm that we call the $q$-\textit{bully path algorithm} on the multiline queues to project them to a word and assign weights to each such projection. The probability of each projection is defined in terms of the weights assigned to them.

 The structure of the paper is as follows. In Section~\ref{sec:bg}, we explain the multispecies PASEP in detail and then state the main result of this paper in Theorem~\ref{th:PASEPcor}. We also define \textit{linked multiline queues} and describe the $q$-bully path algorithm. The proof of Theorem~\ref{th:PASEPcor} is given in Section~\ref{sec:mainproof}.

   \section{Background and Results}\label{sec:bg}
  In a standard Asymmetric Simple Exclusion Process (ASEP) model, particles move along a lattice with a preferential direction, typically from left to right. The parameter $q \in [0,1)$, known as the \textit{asymmetry parameter}, controls the rate at which particles move in the non-preferred direction. In the limiting case where $q = 1$, the model reduces to the \textit{symmetric simple exclusion process} (SSEP), where particles have no directional preference. On the other hand, when $q = 0$, the system becomes a \textit{totally asymmetric simple exclusion process} (TASEP), where particles move exclusively in the preferred direction.  
  
We will now give a formal description of the multispecies PASEP model on a ring. We consider a ring with a finite number of sites; some of which are occupied by $n$ different types of particles. 
The unoccupied sites or \textit{holes} are then assigned the label $n+1$ and are treated as particles with the highest label.
Let $m=(m_1, \ldots, m_{n+1})$ be a tuple of nonnegative integers and let $N=\sum m_i$.  A multispecies PASEP of type $m$ is a Markov process that is defined on a ring with $N$ sites. For each $i \in [n]$, there are $m_i$ particles with label $i$ that occupy the sites of the ring. There are also $m_{n+1}$ holes. Each site can accommodate at most one particle. Let the state-space of the system be denoted by  $\Lambda_m$ 
and the states are given by the cyclic words $\omega=(\omega_k : k\in [N])$, where $\omega_k \in [n+1]$ is the label of the particle at site $k$. The dynamics of the process are as follows. Each particle carries an exponential clock which rings with rate $1$, and the particle exchanges position with the particle on the right whenever the clock rings. Let the particle on the left and the right be labeled $i$ and $j$ respectively. The transition happens with the following rates.

\[ij \rightarrow ji \text{ with rate }
     \begin{cases}
         1, &\text{if }i>j,\\
         q, &\text{if }i<j.
     \end{cases}\]

In this paper, we are interested in studying the correlations of the two particles at the first two sites of the ring for a PASEP of type $\iden=(1,\ldots,1)$ on a ring with $n$ sites. Let $c_{i,j}^q(n)$ denote the probability in the stationary distribution, that the particles labeled $i$ and $j$ occupy the first and second positions, respectively, on the ring $\mathbb{Z}_n$. Note that for $q=0$, the process becomes a TASEP on a finite ring for which the analysis has already been done by Ayyer and Linusson~\cite{ayyerlinusson}. Let $c_{i,j}^0(n)$ denote the probability that the first two sites of the ring are occupied by particles labeled $i$ and $j$, respectively, when $q=0$. This is formulated in the following theorem.

\begin{thm}\cite[Theorem 4.2]{ayyerlinusson}\label{th:q0}  For $i,j \in [n]$, we have
\[c_{i,j}^0(n) = 
\begin{cases} \frac{i-j}{n\binom{n}{2}}; & \text{ if } i>j, \\
\frac{1}{n^2} + \frac{i(n-i)}{n^2(n-1)}; & \text{ if } i=j-1,\\
\frac{1}{n^2}; & \text{ if }i<j-1.
\end{cases}\]
\end{thm}
We generalise Theorem~\ref{th:q0} for arbitrary $q \in [0,1)$ and prove the following main theorem regarding the two-point correlations in this paper.

\begin{thm}\label{th:PASEPcor}
Let $c_{i,j}^q(n)$ be the probability that particles labeled $i$ and $j$ are in the first and the second positions respectively in a PASEP of type $\iden=(1,\ldots,1)$. For $1\leq j < i \leq n$, we have\\
\begin{multline}\label{ci>j}
    c_{i,j}^q(n) = c_{i,j}^0(n) +\frac{(i-j+1)(2j(n-i)+i+j-n-1)([i-j+1]_q-1)}{n^2(n-1)[i-j+1]_q}\\
    -\frac{(i-j+2)(j-1)(n-i)([i-j+2]_q-1)}{n^2(n-1)[i-j+2]_q}  -\frac{j(i-j)(n-i+1)([i-j]_q-1)}{n^2(n-1)[i-j]_q},
 \end{multline}
and for $1\leq i < j \leq n$, we have
\begin{multline}\label{ci<j}
    c_{i,j}^q(n) =  c_{i,j}^0(n)-\frac{(j-i+1)(2i(n-j)+i+j-n-1)q^{(j-i)}}{n^2(n-1)[j-i+1]_q}\\
    +\frac{(i-1)(j-i+2)(n-j)q^{(j-i+1)}}{n^2(n-1)[j-i+2]_q} +
    \begin{cases}
        \frac{i(j-i)(n-j+1)q^{(j-i-1)}}{n^2(n-1)[j-i]_q} & i<j-1\\
        0& i=j-1
        \end{cases},
\end{multline} 
where $[k]_q = 1+q+\cdots+q^{k-1}$ is the q-analog of an integer $k>0$. 
\end{thm}
\begin{rem}\label{rem:qis0}
Note that setting $q=0$ in \eqref{ci>j} and \eqref{ci<j} gives $c_{i,j}^q(n)=c_{i,j}^0(n)$. 
\end{rem}
We prove this result using Martin's multiline process~\cite{jamesmartin} and lumping. Let's start by defining some notation. Consider the tuple $m=(m_1,\ldots,m_{n+1})$ such that $N = m_1+\cdots+m_{n+1}$. To construct a \textit{multiline queue} (MLQ) of type $m$, take a stack of $n$ rings numbered from top to bottom, each with $N$ sites. For every $k \in [n]$, the $k^{th}$ ring has $S_k=m_1+\cdots +m_k$ sites occupied by particles, with no additional constraint. An occupied site is represented by $\bullet$, while an unoccupied site (or hole) is represented by $\circ$. For ease of presentation, we represent the rings as lines with their ends connected. See Figure~\ref{fig:MLQ1} for an example of a multiline queue of type $(2,1,2,2,6)$. Let $\Omega_m$ denote the set of all multiline queues of type $m$. Since the selection of occupied sites on different lines is independent, it is easy to see that the number of multiline queues in $\Omega_m$ is $\binom{N}{S_1}\binom{N}{S_2}\ldots\binom{N}{S_n}$.
\begin{figure}[h!]
   \centering
    \includegraphics[scale=0.4]{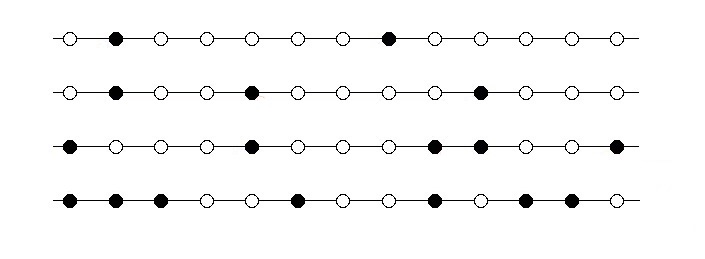}
    \caption{A multiline queue of type $(2,1,2,2,6)$}
    \label{fig:MLQ1}
\end{figure}

 Using the \textit{q-bully path} algorithm from \cite{jamesmartin}, which generalises the bully path algorithm given in \cite{ayyerlinusson, ferrarimartin}, we can project each multiline queue onto many possible words of type $m$. 
The algorithm is a recursive process. 
At each step, we select the topmost row that contains an available particle, link that particle to an available one in the next row, and continue this process until a particle in the last row is linked. The linked particles are then considered unavailable. This process is repeated until every $\bullet$ in the first $n-1$ rows is linked with a $\bullet$ in a row below it. We denote each ``link" between two particles on adjacent rows by an $\rightarrow$, and each link is assigned a weight based on the number of available particles at that step. Corteel, Mandelshtam and Williams~\cite{Intro_CMW_Macdonald} use the term  \textit{pairing} to describe these links. A \textit{linked multiline queue} ($L$MLQ) is defined as a multiline queue along with a maximal set of such links. See Figure~\ref{fig:lMLQ} for an example of a linked multiline queue. 
\begin{figure}[h!]
   \centering
    \includegraphics[scale=0.4]{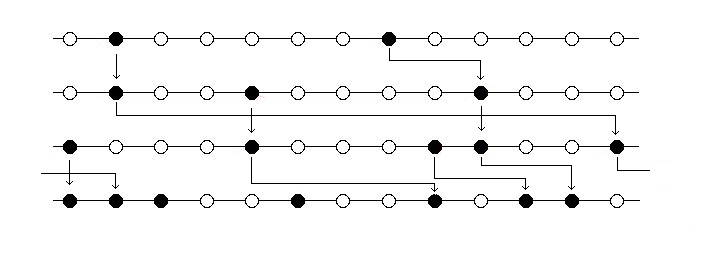}
    \caption{A linked multiline queue of type $(2,1,2,2,6)$}
    \label{fig:lMLQ}
\end{figure}

Each linked multiline queue is associated with both a word and a weight. The weight of an $L$MLQ is defined as the product of the weights of all the links within it. We will first describe the
$q$-bully path algorithm for the case when $n=2$, and then extend the description to the general case of $n$. Let $m=(m_1, m_2, m_3)$, where each site of the multispecies PASEP of type $m$ is either a hole or occupied by a particle of type $1$ or $2$. Consider a multiline queue $M$ of type $m$.

\begin{enumerate}[(1)]
    \item\label{item:bpp1} \textbf{Step 1:} Choose an occupied site $a$ in the first row of $M$.  If there is also a particle at site $a$ in the second row, construct a straight link in the $a^{\text{th}}$ column of $M$ and assign it a weight $1$. This link is referred to as a ``trivial" link.     
    
    \item\label{item:bpp2} \textbf{Step 2:} If there isn’t a particle at site $a$ in the second row, let there be $t$ available particles in the second row at sites $b_1, \ldots, b_t$. Reorder these $t$ particles in increasing order of the values $(b_j-a) \mod N$. The particle at site $a$ can be linked to any of the $t$ particles, resulting in multiple possible $L$MLQs. If it is linked to the particle at site $b_i$, then the link $a \rightarrow b_i$ has a weight given by 
    \[
    \frac{q^{i-1}}{[t]_q}.
    \]
    The particle at site $b_i$ is now unavailable for further linking. Refer to Figure~\ref{fig:links} for examples of links and their corresponding weights. Repeat this process by choosing particles in the first row in any arbitrary order and linking them to particles in the second row. In Example~\ref{exmp:bullyPASEP}, we proceed from left to right.    
    
    \begin{figure}[h]
    \centering
    \includegraphics[scale=0.6]{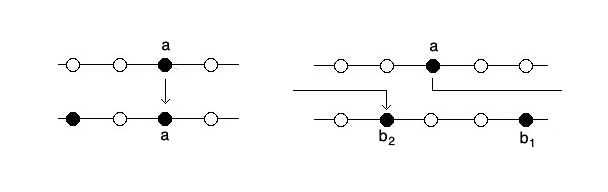}
    \caption{Examples of a trivial and a non-trivial link. The weight of the trivial link on the left is $1$, and the weight of the non-trivial link on the right is $\frac{q^{(2-1)}}{[2]_q}$. }
    \label{fig:links}
\end{figure}
     
    \item \textbf{Step 3:} Label all the linked particles in the second row as type $1$, the unlinked particles as type $2$, and the unoccupied sites as type $3$. This algorithm creates a linked multiline queue of $M$. The associated word $\omega=(\omega_i: i \in [n])$, where $\omega_i$ represents the label of site $i$ in the second row of the multiline queue, is referred to as the \textit{projected word} of the $L$MLQ.
\end{enumerate}
The weight of a linked multiline queue is defined as the product of the weights of all its links. The probability of a word $\omega$ of type $m$ in $\Lambda_m$ is proportional to the sum of weights of all the $L$MLQs that project to $\omega$. Next, we will illustrate the $q$-bully path algorithm with examples of various linked multiline queues derived from the same multiline queue.
\begin{exmp}\label{exmp:bullyPASEP}
Let $m=(3,2,3)$. Let $M$ be a multiline queue of type $m$, i.e.,
{\large 
\[M=\begin{array}{cccccccccc}
        \circ & \circ  & \bullet & \circ & \bullet & \bullet & \circ & \circ \\
        \bullet & \bullet & \circ & \circ & \bullet & \circ & \bullet & \bullet
    \end{array}.\]}
     
\begin{figure}[h!]
    \centering
    \includegraphics[scale=0.5]{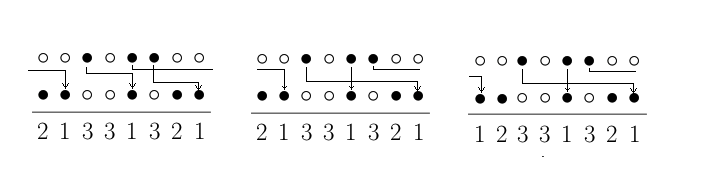}
    \caption{Linked multiline queues of an MLQ of type (3,2,3). }
    \label{fig:bullyconf}
\end{figure}

The multiline queue $M$ can generate many linked multiline queues (LMLQs), three of which are illustrated in Figure~\ref{fig:bullyconf}. The weights of the linked MLQs from left to right are $\frac{1\cdot q^3\cdot q}{[5][4][3]}$, $\frac{q^2 \cdot q^2}{[5][3]}$, and $\frac{q^2 \cdot q}{[5][3]}$ respectively. Note that two or more distinct $L$MLQs can project to the same word but have different weights, as illustrated by the first two examples in Figure~\ref{fig:bullyconf}.\end{exmp}

The $q$-bully path algorithm for $n=2$ can be interpreted in terms of a queueing process. The indices with $\bullet$'s in the first row can be viewed as ''arrival times" in a system of queues, and the indices with $\bullet$'s in the second row represent ''service times." For each arrival time, the algorithm assigns a unique ''departure time" from the available service times. The sites in the second row are thus categorised as departure times, times of unused service, or times of no service. These are labeled as $1$, $2$, and $3$ respectively in the projected word.

Now, we define the algorithm for $n > 2$ recursively.

\begin{enumerate}
    \item \textbf{Step 1:} Let $m = (m_1, \dots, m_{n+1})$ be such that $N = \sum m_i$. Consider $M$, a multiline queue of type $m$. Link all the particles in the first row to $m_1$ particles in the second row by following the $q$-bully path algorithm for $n = 2$ as described earlier. Label the sites in the second row as $1$, $2$, or $3$ accordingly.
    
    \item \textbf{Step 2:} For $k > 1$, assume that all the particles in the $j^{th}$ row have been linked to an available particle in the $(j+1)^{st}$ row for all $j < k$. Each particle in the $k^{th}$ row lies at the end of a chain of links starting in the $i^{th}$ row for some $i \leq k$. Such a particle is said to have type $i$. Repeat steps (1) and (2) of the algorithm for the case $n = 2$ for all the particles in the $k^{th}$ row, in the increasing order of their types; by creating links to the $(k+1)^{st}$ row and assigning a weight to each link according to the number of available particles. Thus, all the particles in the first $k^{th}$ row have now been linked to an available particle in the row below.
    
    \item \textbf{Step 3:} Repeat step $2$ for all $k < n$. This results in a linked multiline queue. Assign each particle in the $n^{th}$ row a label corresponding to its type. The holes are assigned the label $n+1$. The weight of the $L$MLQ is the product of the weights of all the links in it. The labels of all the sites in the $n^{th}$ row generate a word $\omega = (\omega_i: i \in [n+1])$.
    \end{enumerate}

Refer to Figure~\ref{fig:lMLQbully} for an example of a $q$-bully path projection applied to a multiline queue consisting of four rows. In the figure, the product of the link weights for each row is displayed to the right of the corresponding row. The weight of the linked multiline queue can be computed by multiplying together the weights from each row.  As in the case when $n=2$, the probability of the word $\omega$ of type $m$ is proportional to the sum of weights of all the $L$MLQs that project to $\omega$.
\begin{figure}[h]
     \centering
    \includegraphics[scale=0.5]{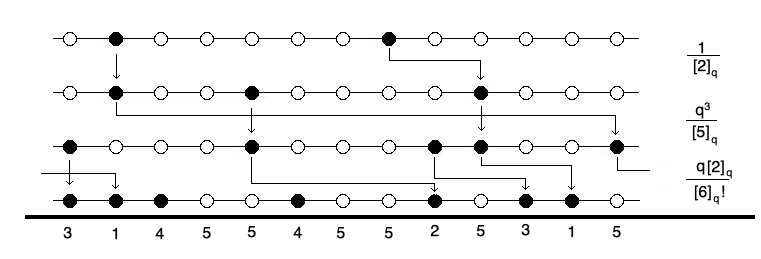}
    \caption{An $L$MLQ of type $(2,1,2,2,6)$ with weight {\large $\frac{q^4}{[6]_q![5]_q}$}}
    \label{fig:lMLQbully}
    \end{figure}
    
Finally, we use a property known as the \textit{projection principle} to prove the main result of this paper. This principle states that particles of other types cannot distinguish particles of two consecutive types.  Recall that  $\iden=(1,\ldots,1)$ and let $\Omega_n$ be the set of all multiline queues of type $\iden$. To compute the correlation $c_{ij}^q(n)$ from Theorem~\ref{th:PASEPcor} for $i,j \in [n]$, we consider all the $L$MLQ of type $\iden$ that project to a word with $i$ in the first position and $j$ in the second.  Using the projection principle, we can identify two consecutive labels $k$ and $k+1$, and define a natural projection from a PASEP with $n$ species to a PASEP with $n-1$ species. To compute the correlation $c_{ij}^q(n)$ for $i>j$, it suffices to find the probability that a $3$ is followed by a $2$ in the projection of the word of a three-species multiline queue of type $m_{s,t}=(s,t,n-s-t)$. This is obtained by the repeated application of the projection principle, allowing us to lump the multispecies PASEP of type $\iden$ to the multispecies PASEP of type $m_{s,t}=(s,t,n-s-t)$. On the other hand for $i<j$, $i$ becomes a $2$, and $j$ becomes a $3$ in the PASEP of type $m_{s,t}$. Let $\Omega_{s,t}(n)$ be the set of multiline queues of type $m_{s,t}$ and let
 \begin{align*}
    T^<_{s,t}(n) =& \mathbb{P}\{\omega_1=2,\, \omega_2=3\},\,\, \text{ and }
    T^>_{s,t}(n) = \mathbb{P}\{\omega_1=3,\, \omega_2=2\}, 
\end{align*}
where $\omega$ 
is a random word in the state space of the multispecies PASEP of type  $m_{s,t}$. Let $i<j$. Then, by the projection principle we have \[T^<_{s,t}(n)=\sum_{j=s+t+1}^n \sum_{i=s+1}^{s+t} c_{ij}^q(n),\] and using the principle of inclusion-exclusion we get \begin{equation}\label{PIE6}
c_{i,j}^q(n)= T^<_{i-1,j-i}(n) - T^<_{i,j-i-1}(n)-T^<_{i-1,j-i+1}(n) +T^<_{i,j-i}(n).
\end{equation}
Similarly, for $i>j$ we have  \[T^>_{s,t}(n)=\sum_{i=s+t+1}^n \sum_{j=s+1}^{s+t} c_{ij}^q(n),\] 
\begin{equation}\label{PIE5}
\text{Thus, } c_{i,j}^q(n)= T^>_{j-1,i-j}(n) - T^>_{j,i-j-1}(n)-T^>_{j-1,i-j+1}(n) +T^>_{j,i-j}(n).
\end{equation}
\section{Proof of Theorem~\ref{th:PASEPcor}}\label{sec:mainproof}
We first prove Theorem~\ref{th:PASEPcor} for the case $i>j$. Let $M$ be a multiline queue of type $m_{s,t}$ and $(r,p)$ be the coordinate of $p^{th}$ site in the $r^{th}$ row such that $M_{r,p} \in \{\circ, \bullet \}$ denotes the occupancy status of the site at $(r,p)$. We use \eqref{PIE5} to solve for $c^q_{i,j}$. To compute $T^>_{s,t}$, we only need to consider the multiline queues in $\Omega_{s,t}(n)$ which have either of the following structures:
 \begin{center}
   \[
   \begin{array}{cccccccccc}
          \circ & \circ & . & . & \ldots & . \\
          \circ & \bullet & . & . & \ldots & . \\
        \hline
          3 & 2 & . & . & \ldots & .
    \end{array} \hspace{0.5in} \text{ or } \hspace{0.5in}  \begin{array}{cccccccccc}
        \bullet & \circ & . & . & \ldots & . \\
        \circ & \bullet & . & . & \ldots & . \\
        \hline
        3 & 2 & . & . & \ldots & .
    \end{array}\] 
 \end{center}
 
This holds because an unoccupied site in a $2$-species system is labeled $3$, hence $M_{2,1}=\circ$ to ensure $\omega_1=3$. Also, $M_{2,2}=\bullet$ so that $\omega_2 \neq 3$. Further, if $M_{1,2}=\bullet$, we have a trivial link at the second site in $M$ giving $\omega_2=1$. Therefore, $M_{1,2}= \circ$.

Before computing the weights contributed by the $L$MLQs in the above two cases, we first consider the set $\Theta_{s,t}(k)$ of multiline queues $M$ of type $(s,t,k-s-t)$ such that there is no $\bullet$ at the same site in both the rows in $M$. In other words, $\Theta_{s,t}(k)$ consists of $L$MLQs with no trivial links. Let $\eta_{s,t}(k)$ be the total weight of all the $L$MLQs in $\Theta_{s,t}(k)$ that project to a word beginning with $2$. By the same argument as in the previous paragraph, this requires that we only consider the multiline queues in $\Theta_{s,t}(k)$ that begin with
($\cb$). Let $C_{s,t}(k)$ be the number of such multiline queues, i.e., $C_{s,t}(k)=|\{M' \in \Theta_{s,t}(k) : M' \text{ begins with } (\cb)\}| $. We have 
\[C_{s,t}(k) = \binom{k-1}{s}\binom{k-s-1}{s+t-1},\]
because the first column is fixed, and we only have to select $s$ ($\bc$) and $s+t-1$ ($\cb$) columns from $k-1$ columns. Next, we compute $\eta_{s,t}(k)$.

\begin{thm}\label{th:etaPASEP}
For $s,t \geq 1$ and $s+t\leq k$, we have
\begin{equation}\label{etaPASEP}
\eta_{s,t}(k)=\frac{t}{s+t}\binom{k-1}{s,\,s+t-1,\,k-2s-t}.
\end{equation}
\end{thm}
\begin{rem}\label{rem:qindependence}
It is interesting to note that despite being a sum of link weights which are $q$-fractions, $\eta_{s,t}(k)$ adds up to a rational number. There is no dependence on $q$. 
\end{rem}
Before looking at the proof of Theorem~\ref{th:etaPASEP}, let us first consider an example for $\Theta_{1,2}(4)$. 
\begin{exmp}\label{exmp:etaexample}
We have $C_{1,2}(4)=3$. We consider below the three multiline queues from the set $\Theta_{1,2}(4)$ that begin with ($\cb$) and list out all the possible projected words along with their weights. 
According to Theorem~\ref{th:etaPASEP}, \[\eta_{1,2}(4)=\frac{2}{3} C_{1,2}(4) =\,2.\]

\[\begin{array}{cccccc}
      \circ & \bullet & \circ & \circ & &\\
      \bullet & \circ & \bullet & \bullet & & wt\\
      \hline
     \color{red}{2} & \color{red}{3} & \color{red}{1} & \color{red}{2} & &\color{red}{\frac{1}{[3]_q}}\\
      \color{red}{2} & \color{red}{3} & \color{red}{2} & \color{red}{1} & &\color{red}{\frac{q}{[3]_q}}\\
      1 & 3 & 2 & 2 & &\frac{q^2}{[3]_q}
\end{array}\hspace{0.6in}
 \begin{array}{cccccc}
      \circ  & \circ & \bullet & \circ & &\\
      \bullet  & \bullet & \circ & \bullet & &wt \\
      \hline
      \color{red}{2} & \color{red}{2} & \color{red}{3} & \color{red}{1} & &\color{red}{\frac{1}{[3]_q}}\\
      1 & 2 & 3 & 2 & &\frac{q}{[3]_q}\\
      \color{red}{2} & \color{red}{1} & \color{red}{3} & \color{red}{2} & &\color{red}{\frac{q^2}{[3]_q}}\\
\end{array}\hspace{0.6in}
 \begin{array}{cccccc}
      \circ  & \circ & \circ & \bullet & \\
      \bullet  & \bullet & \bullet & \circ &wt\\
      \hline
      1 & 2 & 2 & 3 & \frac{1}{[3]_q}\\
      \color{red}{2} & \color{red}{1} & \color{red}{2} & \color{red}{3}  &\color{red}{\frac{q}{[3]_q}}\\
      \color{red}{2} & \color{red}{2} & \color{red}{1} & \color{red}{3} & \color{red}{\frac{q^2}{[3]_q}}\\
\end{array}\]

Note that the words in red are the ones that begin with a $2$ and the sum of the weights of $L$MLQs in $\Theta_{1,2}(4)$ that project to such a word is $\eta_{1,2}(4)=\frac{2(1+q+q^2)}{[3]_q}=2$.

In addition, we make the following observations about this example. Adding the weights of all possible projected words is $1$ for all three multiline queues. Next, rotating a linked multiline queue rotates the projected word by the same distance while preserving the weight. In each row above, the configurations in each column are rotations of one another and have the same weight. Based on these observations, we prove these properties for a more general class of multiline queues.
\end{exmp}

\begin{lem}\label{lem:rotation}
Let $M$ be a multiline queue in $\Theta_{s,t}(k)$. The following holds for $M$.
\begin{enumerate}[(1)]
    \item\label{item:lemmasum} The sum of weights of all the linked multiline queues of $M$ is $1$.
    \item\label{item:lemmarotn} Rotating $M$ while keeping the same links rotates the projected word by the same distance, while the weights of the corresponding linked multiline queues remain unchanged. 
\end{enumerate}
\end{lem}
 \begin{proof}
  We prove~(\ref{item:lemmasum}) by induction on $s$. Let $s=1$. There are $t+1$ possible links from the only $\bullet$ in the first row to a particle in the second row. This accounts for $t+1$ $L$MLQs of $M$; each corresponding to one of these possibilities and they have respective weights $q^{i-1}/[t+1]_q$ for $i \in \{1,\ldots, t+1\}$. Adding these weights for all $i$, we get $1$.
  
  Let us assume~(\ref{item:lemmasum}) is true for $s-1$. Let $M$ have $s$ $\bullet$'s in the first row and $(s+t)$ $\bullet$'s in the second row at sites different from those with $\bullet$'s in the first row. Let the occupied sites in the second row be labeled as $b_1,\ldots,b_{s+t}$. We can link the particles of the first row to the particles in the second row in any order, in particular from left to right. Let the leftmost $\bullet$ (say at site $a$) in the first row be linked to the particle at site $b_j$, for some $j$ and let the weight of this link be $q^{i-1}/[s+t]_q$, where $i=j-c \mod (s+t)$ for some constant $c$. Constructing links for the remaining $\bullet$'s in the first row is the same as constructing links in a multiline queue $M_d$ which is obtained from $M$ by deleting the columns $a$ and $b_j$. Note that $M_d \in \Theta_{s-1,t}(k-2)$ and the sum of weights of all the $L$MLQs of $M_d$ is $1$. That is, the sum of all the $L$MLQs of $M$ where $a \rightarrow b_j$ is $1\cdot(q^{i-1}/[s+t]_q)$. Summing over all $j \in [s+t]$ (equivalently over all $i \in [s+t]$), we get$~(\ref{item:lemmasum}$).
  
  To prove~(\ref{item:lemmarotn}), recall that the projected word describes the label of each site in the second row. Hence, rotating both rows of the MLQ simultaneously while keeping the links preserved only rotates the projected word without changing the weights of any links. 
 \end{proof}
 
\begin{proof} [Proof of Theorem~\ref{th:etaPASEP}]
 Let $\mathcal{S}$ be the set of all the linked multiline queues in $\Theta_{s,t}(k)$ which begin with ($\cb$) and project to a word beginning with $2$. Recall that $\eta_{s,t}(k)$ is the sum of the weights of all $L$MLQs in $\mathcal{S}$.  Let $P \in \mathcal{S}$ has the following structure 
 \[P= \begin{array}{cccccccccc}
        \circ  & . & . & \ldots & . \\
        \bullet & . & . & \ldots & . \\
        \hline
         2 & . & . & \ldots & .
\end{array}. \]
    Further, recall that $C_{s,t}(k)$ is the cardinality of set \[\mathcal{M}=\{M \in \Theta_{s,t}(k): M \text{ begins with } (\cb) \}.\] Consider any $M \in \mathcal{M}$. The weights from all the $L$MLQs of $M$ sum up to $1$ by Lemma~\ref{lem:rotation} (\ref{item:lemmarotn}). Note that not all of these $L$MLQs belong to $\mathcal{S}$, as some project to a word with a $1$ in the first position. However, each $L$MLQ of $M$ has exactly $t$ rotations that are in $\mathcal{S}$ because there are $t$ $\bullet$'s in the second row that are not linked. Let $\Gamma$ be the collection of all the rotations of all the linked multiline queues in $\mathcal{M}$ that belong to $\mathcal{S}$. Note that $\Gamma$ is a multiset. 
By Lemma~\ref{lem:rotation} (\ref{item:lemmasum} and \ref{item:lemmarotn}),
 the sum of the weights of all the linked multiline queues in $\Gamma$ is $tC_{s,t}$,
 
 Further, note that for each $L$MLQ in $\mathcal{S}$, there are $(s+t)$ rotations which start with $(\cb)$. In other words, each configuration of $\mathcal{S}$ is obtained as a rotation of $(s+t)$ different $L$MLQs in $\mathcal{M}$. Therefore, each linked multiline queue of $\mathcal{S}$ has $(s+t)$ copies in $\Gamma$. Hence, the sum of the weights of all the linked multiline queues in $\Gamma$ is $(s+t)\eta_{s,t}(k)$. This gives us the equation \begin{equation}
    (s+t)\eta_{s,t}(k) = tC_{s,t}(k),
\end{equation} thereby completing the proof.

\end{proof}

 To compute $T^>_{s,t}(n)$, recall that we only need to consider the $L$MLQs with either of the following structures: 
\[\text{(A) \, } \begin{array}{cccccccccc}
        \circ & \circ & . & . & \ldots & . \\
        \circ & \bullet & . & . & \ldots & . \\
        \hline
        3 & 2 & . & . & \ldots & .
    \end{array}  \hspace{0.3in} \text{ or } \hspace{0.3in}
    \text{(B) \, }\begin{array}{cccccccccc}
        \bullet & \circ & . & . & \ldots & . \\
        \circ & \bullet & . & . & \ldots & . \\
        \hline
        3 & 2 & . & . & \ldots & .
    \end{array}\]
    For \( X \in \{A, B\} \), let \( W^X_{s,t}(n) \) denote the total weight of the linked multiline queues of type \( (X) \) in \( \Omega_{s,t}(n) \) where the corresponding word begins with \( (3,2) \). Next, let \( U^X_{s,t}(n) \subset \Omega_{s,t}(n) \) represent the set of LMLQs of type \( (X) \) that do not have a \( (\bb) \) column. Finally, let \( \tau^X_{s,t}(n) \) be the weight contributed to \( W^X_{s,t}(n) \) by the LMLQs in \( U^X_{s,t}(n) \).
    
    \begin{lem}\label{lem:alphatype1} Let $s,t\geq 0$, $n>s+t$, and $X\in \{A,B\}$. $W^X_{s,t}$ and $\tau^X_{s,t}$ are related by the following equations.
    \begin{eqnarray}\label{alpharec}
    W^A_{s,t}(n)=\sum_{i=0}^s\binom{n-2}{i} \tau^A_{s-i,t}(n-i),\\
    W^B_{s,t}(n)=\sum_{i=0}^{s-1}\binom{n-2}{i} \tau^B_{s-i,t}(n-i)\label{betarec}.
    \end{eqnarray}
    \end{lem}
    \begin{proof}
     Links in any $L$MLQ can be constructed in any arbitrary order. Therefore, we first construct all the trivial links and then process the remaining particles from left to right. Since the weight of a trivial link is $1$, 
     the weight of an $L$MLQ is equal to the product of the weights of non-trivial links.
     
     Consider an arbitrary linked multiline queue contributing to $W^X_{s,t}(n)$ that has $i$ columns of the form ($\bb$). Then, $0 \leq i \leq s$ for $X=A$ and $0 \leq i \leq s-1$ for $X=B$. There are $\binom{n-2}{i}$ ways to choose these $i$ columns. Deleting these columns results in an $L$MLQ of type $(s-i,t,n-i)$ that belongs to $U^X_{s-i,t}(n-i)$. 
     Summing over all possible values of $i$, we get Lemma~\ref{lem:alphatype1}.
    \end{proof}
    Note that the linked multiline queues contributing to $\tau^A_{s,t}(n)$ are in bijection with those of $\eta_{s,t}(n-1)$ in $\Theta_{s,t}(n-1)$ because the first column $(\cc)$ of a multiline queue of type ($A$) does not contribute to the weight of the configuration. Therefore, $\tau^A_{s,t}(k)=\eta_{s,t}(k-1)$ for all $k>1$. Substituting this equation into~\eqref{alpharec}, we get:
    \begin{eqnarray*}
    W^A_{s,t}(n) & = & \sum_{i=0}^s \binom{n-2}{i} \eta_{s-i,t}(n-i-1)\\
    & = &\frac{t(n-2)!(n-1)!}{s!(n-s-1)!(s+t)!(n-s-t-1)!}.
    \end{eqnarray*}
    Therefore,
     \[\frac{W^A_{s,t}(n)}{\binom{n}{s}\binom{n}{s+t}}=\frac{t(n-s)(n-s-t)}{n^2(n-1)}.\]
     
 

    To compute,  $ \tau^B_{s,t} $, we start by creating links from the second particle in the first row and proceed rightwards along the ring. Finally, we link the $\bullet $ at the first site with one of the $t + 1$  available particles in the second row, excluding the first one. Thus, the weight of this link is $ \frac{q^{h-1}}{(t + 1)q} $ for some  $h \in \{2, \ldots, t + 1\}$. Removing the first column from any linked multiline queue in  $U^B_{s,t}(n)$  results in a linked multiline queue of type $(s - 1, t + 1, n - 1)$ that projects to a word beginning with 2 and has no $(\bb)$ column. 

       Recall that the sum of weights of these $L$MLQs is equal to $\eta_{s-1,t+1}(n-1)$. Therefore,
    \begin{eqnarray*}
    \tau^B_{s,t}(n)&=& \eta_{s-1,t+1}(n-1) \sum_{h=2}^{t+1} \frac{q^{h-1}}{[t+1]_q} \\
    &=&\frac{t+1}{s+t}\binom{n-2}{s-1} \binom{n-s-1}{s+t-1} \Big(1-\frac{1}{[t+1]_q}\Big). 
    \end{eqnarray*}

        
    Substituting this in $\eqref{betarec}$,
     \begin{eqnarray*}
    W^B_{s,t}(n) & = & \sum_{i=0}^{s-1} \binom{n-2}{i} \tau^B_{s-i,t}(n-i)\\
    & = &\frac{(t+1)(n-2)!(n-1)!}{(s-1)!(n-s)!(s+t)!(n-s-t-1)!}\Big(1-\frac{1}{[t+1]_q}\Big).
    \end{eqnarray*}
    Therefore,
     \[\frac{W^B_{s,t}(n)}{\binom{n}{s}\binom{n}{s+t}}=\frac{s(t+1)(n-s-t)}{n^2(n-1)}\Big(1-\frac{1}{[t+1]_q}\Big).\]

    We have \[T^>_{s,t}(n)=\frac{W^A_{s,t}(n)+W^B_{s,t}(n)}{\binom{n}{s}\binom{n}{s+t}}.\] If $t=0$ or $s+t=n$, the formula is trivially satisfied in either case because then, the projected word cannot start with $(3,2)$. Consider the case when $s=0$. There are no type ($B$) multiline queues. For a  type $(A)$ multiline queue, the first row entirely consists of $\circ$, and the second row begins with $\circ \, \bullet \, . \, \ldots \, .$, with $t-1$ additional $\bullet$'s. Thus, there are $\binom{n-2}{t-1}$  such multiline queues, each contributing a weight $1$. So, $T^>_{0,t}(n)=\frac{\binom{n-2}{t-1}}{\binom{n}{t}} = \frac{t(n-t)}{n(n-1)}$. Hence, we can write 
    \begin{align}
        T^>_{s,t}(n)&=\frac{t(n-s)(n-s-t)}{n^2(n-1)} + \frac{s(t+1)(n-s-t)}{n^2(n-1)} \Big(1-\frac{1}{[t+1]_q}\Big) \nonumber\\
        &= \frac{t(n-s)(n-s-t)}{n^2(n-1)}+ \frac{s(t+1)(n-s-t)([t+1]_q-1)}{n^2(n-1)[t+1]_q}, \label{Tst1}
    \end{align}
    for all $0 \leq s,t \leq n, s+t \leq n$.
    \begin{rem}\label{qindependence}
    Interestingly, the weight contributed by the type ($A$) multiline queues does not depend on the value of $q$. Further, the formula in \eqref{Tst1} is consistent with $T^>_{s,t}(n)$ in \cite{ayyerlinusson} for the case $q=0$, i.e., the multispecies TASEP on a ring.   
    \end{rem} 
    The proof is completed by substituting ~\eqref{Tst1} in~\eqref{PIE5}. 
    We get 
    \begin{multline}\label{cij3}
    c_{i,j}^q(n) = \frac{2(i-j)}{n^2(n-1)} +\frac{(i-j+1)(2j(n-i)+i+j-n-1)([i-j+1]_q-1)}{n^2(n-1)[i-j+1]_q}\\
    -\frac{(i-j+2)(j-1)(n-i)([i-j+2]_q-1)}{n^2(n-1)[i-j+2]_q}  -  \frac{j(i-j)(n-i+1)([i-j]_q-1)}{n^2(n-1)[i-j]_q},
\end{multline} when $i>j$.
This proves Theorem~\ref{th:PASEPcor} for the case $1 \leq j < i \leq n.$\\

 The case $i<j$ is now resolved similarly. Recall that $T^<_{s,t}(n)$ denotes the probability $\mathbb{P}\{\omega_1=2, \,\omega_2=3\}$, where $\omega$ is the word projected by a random multiline queue in $\Omega_{s,t}(n)$. The multiline queues that project to a word beginning with ($2,3$) are of one of the following types.
\[\text{(C) } \begin{array}{cccccccccc}
        \circ & \circ & . & . & \ldots & . \\
        \bullet & \circ & . & . & \ldots & . \\
        \hline
        2 & 3 & . & . & \ldots & .
    \end{array} \hspace{0.3in} \text{ or } \hspace{0.3in} \text{ (D) }
    \begin{array}{cccccccccc}
        \circ & \bullet & . & . & \ldots & . \\
        \bullet & \circ & . & . & \ldots & . \\
        \hline
        2 & 3 & . & . & \ldots & .
    \end{array}\]
We compute the weights of the two cases separately as follows:

\begin{enumerate}[(A)]\addtocounter{enumi}{+2}
     \item 
    Let $W^C_{s,t}(n)$ denote the sum of the weights of the $L$MLQs of MLQs of type ($C$) in $\Omega_{s,t}(n)$ for which the corresponding word begins with $(2,3)$. 
    Note that interchanging the first two columns of an $L$MLQ of type $(C)$ does not change its weight. Hence, $W^C_{s,t}(n)=W^A_{s,t}(n)$.
    \item 
    Let $W^D_{s,t}(n)$ denote the total weight of the linked multiline queues of type $(D)$. Let $U^D_{s,t}(n)$ be the set of the $L$MLQs that do not contain any ($\bb$) column. 
    Finally, let $\tau^D_{s,t}(n)$ represent the weight contributed to $W^D_{s,t}$ by the linked multiline queues in $U^D_{s,t}(n)$. There is a similar relation between $W^D_{s,t}(n)$ and $\tau^D_{s,t}(n)$ as there is for $W^B_{s,t}(n)$ and $\tau^B_{s,t}(n)$, and this can be proved using the same arguments as in Lemma~\ref{lem:alphatype1}. We have
 \end{enumerate}
    \begin{equation}\label{betadrec}
    W^D_{s,t}=\sum_{i=0}^{s-1}\binom{n-2}{i} \tau^D_{s-i,t}(n-i).
    \end{equation}
     To determine $\tau^D_{s,t}$, we start by creating links from the second $\bullet$ in the first row and move rightwards along the ring. We link leftmost $\bullet$ in the first row last, and this link has a weight of  $q^{h-1}/[t+1]_q$ for some $h \in \{1,\ldots, t\}$. This is because, for $h=t+1$, we get $\omega_1=1$. Removing the second column from any $L$MLQ in $U^D_{s,t}(n)$ results in an $L$MLQ of type $(s-1,t+1,n-1)$ that begins with $(\cb)$, has no $(\bb)$ column, and projects to a word that starts with $2$. The sum of weights of these linked multiline queues is again equal to $\eta_{s-1,t+1}(n-1)$.
     Therefore,
    \begin{eqnarray*}
    \tau^D_{s,t}(n)  & = &\eta_{s-1,t+1}(n-1) \sum_{h=1}^{t} \frac{q^{h-1}}{[t+1]_q}\\
    &=&\frac{t+1}{s+t}\binom{n-2}{s-1} \binom{n-s-1}{s+t-1} \Big(1-\frac{q^t}{[t+1]_q}\Big).
    \end{eqnarray*}
    Substituting this in \eqref{betadrec}, we get
     \begin{eqnarray*}
    W^D_{s,t}(n) & = & \sum_{i=0}^{s-1} \binom{n-2}{i} \tau^D_{s-i,t}(n-i)\\
    & = &\frac{(t+1)(n-2)!(n-1)!}{(s-1)!(n-s)!(s+t)!(n-s-t-1)!}\Big(1-\frac{q^t}{[t+1]_q}\Big).
    \end{eqnarray*}
Therefore, 
\[\frac{W^D_{s,t}(n)}{\binom{n}{s}\binom{n}{s+t}}=\frac{s(t+1)(n-s-t)}{n^2(n-1)}\Big(1-\frac{q^t}{[t+1]_q}\Big).\]
Once again we have, \[T^<_{s,t}=\frac{W^C_{s,t}+W^D_{s,t}}{\binom{n}{s}\binom{n}{s+t}},\] that is, for $t>0$ and $n>s+t$, 
\begin{equation}\label{Tst2}
T^<_{s,t}(n)=\frac{(s+tn)(n-s-t)}{n^2(n-1)} - \frac{s(t+1)(n-s-t)q^t}{n^2(n-1)[t+1]_q},
\end{equation}
and substituting this in \eqref{PIE6}, we get
\begin{multline}\label{cij4}
    c_{i,j}^q(n) = \frac{1}{n^2} -\frac{(j-i+1)(2i(n-j)+i+j-n-1)q^{(j-i)}}{n^2(n-1)[j-i+1]_q}\\
     +\frac{i(j-i)(n-j+1)q^{(j-i-1)}}{n^2(n-1)[j-i]_q}+\frac{(i-1)(j-i+2)(n-j)q^{(j-i+1)}}{n^2(n-1)[j-i+2]_q},
\end{multline} 
when $i<j-1$.
For $i=j-1$, we add terms corresponding to $T^<_{i,j-i-1}$ from \eqref{Tst2} to \eqref{cij4}, which are 
\[\frac{i(n-i)}{n^2(n-1)}-\frac{i(j-i)(n-i)q^{(j-i-1)}}{n^2(n-1)[j-i]^q}.\]  This proves Theorem~\ref{th:PASEPcor} for the case $1 \leq i < j \leq n$.

\section*{Acknowledgements}
 
I am thankful to my advisor Professor Arvind Ayyer for suggesting the problem and all the insightful discussions during the preparation of this paper and to Professor Leonid Petrov for suggesting a correction in the main theorem of the paper. I am also grateful to a SERB grant CRG/2021/001592 for partial support.

 \bibliography{mPASEP.bib}
\bibliographystyle{plain.bst}

\end{document}